\documentclass[a4paper,pdftex]{amsart}
\usepackage{amsmath}
\usepackage{amssymb}

\newtheorem{theorem}{Theorem}
\newtheorem{definition}{Definition}
\newtheorem{lemma}{Lemma}
\newtheorem{example}{Example}
\newtheorem{corollary}{Corollary}
\newtheorem{remark}{Remark}

\newcommand{\gauss}[3]{\genfrac{[}{]}{0pt}{}{#1}{#2}_{#3}}
\newcommand{\vek}[1]{\mathbf{#1}}
\newcommand{\nequiv}{\not\equiv}

\DeclareMathOperator{\PGammaL}{P\Gamma L}
\DeclareMathOperator{\GL}{GL}
\DeclareMathOperator{\PGL}{PGL}
\DeclareMathOperator{\GF}{GF}
\DeclareMathOperator{\Aut}{Aut}
\DeclareMathOperator{\id}{id}
\DeclareMathOperator{\STS}{STS}
\DeclareMathOperator{\PG}{PG}

\begin{document}

\title[Order of the automorphism group of a $q$-analog of the Fano plane]{The order of the automorphism group of a binary $\mathbf{q}$-analog of the Fano plane is at most two}
\thanks{The second and the third author  were  
supported  in  part  by  the  project  \textit{Integer  Linear  Programming Models for Subspace Codes and Finite Geometry}  
(KU 2430/3-1, WA 1666/9-1)
from the German Research Foundation.
}


\author{Michael Kiermaier \and Sascha Kurz \and Alfred Wassermann}


\address{M. Kiermaier 
              Mathematisches Institut, Universit\"{a}t Bayreuth, 95447 Bayreuth, Germany  \\
	      http://www.mathe2.uni-bayreuth.de/michaelk/}
\email{michael.kiermaier@uni-bayreuth.de}
\address{S. Kurz 
              Mathematisches Institut, Universit\"{a}t Bayreuth, 95447 Bayreuth, Germany}
\email{sascha.kurz@uni-bayreuth.de}
\address{A. Wassermann 
              Mathematisches Institut, Universit\"{a}t Bayreuth, 95447 Bayreuth, Germany  
}
\email{alfred.wassermann@uni-bayreuth.de}

\maketitle

\begin{abstract}
It is shown that the automorphism group of a binary $q$-analog of the Fano plane is either trivial or of order $2$.
\end{abstract}

\section{Introduction}
\label{intro}
Motivated by the connection to network coding, $q$-analogs of designs have received an increased interest lately.
Arguably the most important open problem in this field is the question for the existence of a $q$-analog of the Fano plane \cite{ES16}.
Its existence is open over any finite base field $\GF(q)$.
The most important single case is the binary case $q=2$, as it is the smallest one.
Nonetheless, so far the binary $q$-analog of the Fano plane has withstood all computational or theoretical attempts for its construction or refutation. 

Following the approach for other notorious putative combinatorial objects as, e.g., a projective plane of order $10$ or a 
self-dual binary $[72,36,16]$ code, the possible automorphisms of a binary $q$-analog of the Fano plane have been investigated in \cite{BKN16}.
As a result~\cite[Theorem~1]{BKN16}, its automorphism group is at most of order $4$, and up to conjugacy in $\GL(7,2)$ it is 
represented by a group in the following list:
\begin{enumerate}
\item[(a)] The trivial group.
\item[(b)] The group of order $2$
\[
	G_2 = \left\langle\left(
	\begin{smallmatrix}
	0& 1& 0& 0& 0& 0& 0 \\
	1& 0& 0& 0& 0& 0& 0 \\
	0& 0& 0& 1& 0& 0& 0 \\
	0& 0& 1& 0& 0& 0& 0 \\
	0& 0& 0& 0& 0& 1& 0 \\
	0& 0& 0& 0& 1& 0& 0 \\
	0& 0& 0& 0& 0& 0& 1 
	\end{smallmatrix}\right)\right\rangle\text{.}
\]
\item[(c)] One of the following two groups of order $3$:
\[
G_{3,1} = \left\langle
\left(
\begin{smallmatrix}
0& 1& 0& 0& 0& 0& 0\\ 
1& 1& 0& 0& 0& 0& 0\\ 
0& 0& 0& 1& 0& 0& 0\\
0& 0& 1& 1& 0& 0& 0\\
0& 0& 0& 0& 0& 1& 0\\
0& 0& 0& 0& 1& 1& 0\\
0& 0& 0& 0& 0& 0& 1
\end{smallmatrix}
\right)
\right\rangle
\quad
\text{and}
\quad
G_{3,2} = \left\langle
\left(
\begin{smallmatrix}
0& 1& 0& 0& 0& 0& 0\\ 
1& 1& 0& 0& 0& 0& 0\\ 
0& 0& 0& 1& 0& 0& 0\\ 
0& 0& 1& 1& 0& 0& 0\\ 
0& 0& 0& 0& 1& 0& 0\\ 
0& 0& 0& 0& 0& 1& 0\\ 
0& 0& 0& 0& 0& 0& 1
\end{smallmatrix}
\right)\right\rangle\text{.}
\]
\item[(d)] The cyclic group of order $4$
\[
G_4 = \left\langle
\left(
\begin{smallmatrix}
1&1&0&0&0&0&0 \\
0&1&1&0&0&0&0 \\
0&0&1&0&0&0&0 \\
0&0&0&1&1&0&0 \\
0&0&0&0&1&1&0 \\
0&0&0&0&0&1&1 \\
0&0&0&0&0&0&1
\end{smallmatrix}\right)
\right\rangle\text{.}
\]
\end{enumerate}

For the groups of order $2$, the above result was achieved as a special case of a more general result on restrictions of the automorphisms of order $2$ of a binary $q$-analog of Steiner triple systems \cite[Theorem~2]{BKN16}.
All the remaining groups have been excluded computationally applying the method of Kramer and Mesner.

In this article, we will extend these results as follows.
In Section~\ref{sect:o3} automorphisms of order $3$ of general binary $q$-analogs of Steiner triple systems $\STS_2(v)$ will be investigated.
The main result is Theorem~\ref{thm:o3}, which excludes about half of the conjugacy types of elements of order $3$ in $\GL(v,2)$ as the automorphism of an $\STS_2(v)$.
In the special case of ambient dimension $7$, the group $\GL(7,2)$ has $3$ conjugacy types $G_{3,1}$, $G_{3,2}$ and $G_{3,3}$ of subgroups of order $3$.
Theorem~\ref{thm:o3} shows that the group $G_{3,2}$ is not the automorphism group of a binary $q$-analog of the Fano plane.
Furthermore, Theorem~\ref{thm:o3} provides a purely theoretical argument for the impossibility of $G_{3,3}$, which previously has been shown computationally in \cite{BKN16}.

In Section~\ref{sect:comp}, the groups $G_4$ and $G_{3,1}$ will be excluded computationally by showing that the Kramer-Mesner equation system does not have a solution.
Both cases are fairly large in terms of computational complexity.
To bring the problems to a feasible level, the solution process is parallelized and executed on the high performance Linux cluster of the University of Bayreuth.
For the latter and harder case $G_{3,1}$, we additionally make use of the inherent symmetry of the search space given by the normalizer of the prescribed group, see also~\cite{kaski2005}.

Finally, the combination of the results of Sections~\ref{sect:o3} and~\ref{sect:comp} yields
\begin{theorem}
	The automorphism group of a binary $q$-analog of the Fano plane is either trivial or of order $2$.
	In the latter case, up to conjugacy in $\GL(7,2)$ the automorphism group is represented by
\[
	\left\langle\left(
	\begin{smallmatrix}
	0& 1& 0& 0& 0& 0& 0 \\
	1& 0& 0& 0& 0& 0& 0 \\
	0& 0& 0& 1& 0& 0& 0 \\
	0& 0& 1& 0& 0& 0& 0 \\
	0& 0& 0& 0& 0& 1& 0 \\
	0& 0& 0& 0& 1& 0& 0 \\
	0& 0& 0& 0& 0& 0& 1 
	\end{smallmatrix}\right)\right\rangle\text{.}
\]
\end{theorem}

\section{Preliminaries}
Throughout the article, $V$ is a vector space over $\GF(2)$ of finite dimension $v$.

\subsection{The subspace lattice}
For simplicity, a subspace of $V$ of dimension $k$ will be called a \emph{$k$-subspace}.
The set of all $k$-subspaces of $V$ is called the \emph{Grassmannian} and is denoted by $\gauss{V}{k}{q}$.
As in projective geometry, the $1$-subspaces of $V$ are called \emph{points}, the $2$-subspaces \emph{lines} and the $3$-subspaces \emph{planes}.
Our focus lies on the case $q = 2$, where the $1$-subspaces $\langle\mathbf{x}\rangle_{\GF(2)}\in \gauss{V}{1}{2}$ are in one-to-one correspondence with the nonzero vectors $\mathbf{x} \in V\setminus\{\mathbf{0}\}$.
The number of all $r$-subspaces of $V$ is given by the Gaussian binomial coefficient
\[
\#\gauss{V}{k}{q}
= \gauss{v}{k}{q}
= \begin{cases}
	\frac{(q^v-1)\cdots(q^{v-r+1}-1)}{(q^r-1)\cdots(q-1)} & \text{if } k\in\{0,\ldots,v\}\text{;}\\
	0 & \text{otherwise.}
\end{cases}
\]
The set $\mathcal{L}(V)$ of all subspaces of $V$ forms the subspace lattice of $V$.
There are good reasons to consider the subset lattice as a subspace lattice over the unary {\lq\lq}field{\rq\rq} $\GF(1)$ \cite{Cohn-2004}.

By the fundamental theorem of projective geometry, for $v \geq 3$ the automorphism group of $\mathcal{L}(V)$ is given by the natural action of $\PGammaL(V)$ on $\mathcal{L}(V)$.
In the case that $q$ is prime, the group $\PGammaL(V)$ reduces to $\PGL(V)$, and for the case of our interest $q = 2$, it reduces further to $\GL(V)$.
After a choice of a basis of $V$, its elements are represented by the invertible $v\times v$ matrices $A$, and the action on $\mathcal{L}(V)$ is given by the vector-matrix-multiplication $\mathbf{v} \mapsto \mathbf{v} A$.

\subsection{Designs}
\begin{definition}
Let $t,v,k$ be integers with $0 \leq t \leq k\leq v$ and $\lambda$ another positive integer.
A set $D \subseteq\gauss{V}{k}{q}$ is called a $t$-$(v,k,\lambda)_q$ \emph{subspace design} if each $t$-subspace of $V$ is contained in exactly $\lambda$ elements (called \emph{blocks}) of $D$.
When $\lambda = 1$, $D$ is called a \emph{$q$-Steiner system}.
If additionally $t=2$ and $k=3$, $D$ is called a \emph{$q$-Steiner triple system} and denoted by $\STS_q(v)$.
\end{definition}

Classical combinatorial designs can be seen as the limit case $q=1$ of a design over a finite field.
Indeed, quite a few statements about combinatorial designs have a generalization to designs over finite fields, such that the case $q = 1$ reproduces the original statement \cite{BKKL-2015-arXiv,KL15,KP15,NP15}.

One example of such a statement is the following \cite[Lemma~4.1(1)]{S90}:
If $D$ is a $t$-$(v, k, \lambda)_q$ design, then $D$ is also an $s$-$(v,k,\lambda_s)_q$ for all $s\in\{0,\ldots,t\}$, where
\[
	\lambda_s := \lambda \frac{\gauss{v-s}{t-s}{q}}{\gauss{k-s}{t-s}{q}}.
\]
In particular, the number of blocks in $D$ equals
\[
\#D = \lambda_0 = \lambda \frac{\gauss{v}{t}{q}}{\gauss{k}{t}{q}}.
\]
	So, for a design with parameters $t$-$(v, k, \lambda)_q$, the numbers $\lambda \gauss{v-s}{t-s}{q}/\gauss{k-s}{t-s}{q}$ necessarily are integers for all $s\in\{0,\ldots,t\}$ (\emph{integrality conditions}).
In this case, the parameter set $t$-$(v,k,\lambda)_q$ is called \emph{admissible}.
It is further called \emph{realizable} if a $t$-$(v,k,\lambda)_q$ design actually exists.

For designs over finite fields, the action of $\Aut(\mathcal{L}(V)) \cong \PGammaL(V)$ on $\mathcal{L}(V)$ provides a notion of isomorphism.
Two designs in the same ambient space $V$ are called \emph{isomorphic} if they are contained in the same orbit of this action (extended to the power set of $\mathcal{L}(V)$).
The \emph{automorphism group} $\Aut(D)$ of a design $D$ is its stabilizer with respect to this group action.
If $\Aut(D)$ is trivial, we will call $D$ \emph{rigid}.
Furthermore, for $G \leq \PGammaL(V)$, $D$ will be called $G$-invariant if it is fixed by all elements of or equivalently, if $G\leq \Aut(D)$.
Note that if $D$ is $G$-invariant, then $D$ is also $H$-invariant for all subgroups $H \leq G$. 

\subsection{Steiner triple systems}
For an $\STS_q(v)$ we have
\begin{align*}
	\lambda_1 & = \frac{\gauss{v-1}{2-1}{q}}{\gauss{3-1}{2-1}{q}} = \frac{q^{v-1} - 1}{q^2-1}\quad\text{and} \\
	\lambda_0 & = \frac{\gauss{v}{2}{q}}{\gauss{3}{2}{q}} = \frac{(q^v - 1)(q^{v-1} - 1)}{(q^3 - 1)(q^2 - 1)}\text{.}
\end{align*}
As a consequence, the parameter set of an ordinary or a $q$-analog Steiner triple system $\STS_q(v)$ is admissible if and only if $v\equiv 1,3\bmod 6$ and $v\geq 3$.
For $q=1$, the existence question is completely answered by the result that a Steiner triple system is realizable if and only if it is admissible \cite{Kirkman-1847}.
However in the $q$-analog case, our current knowledge is quite sparse.
Apart from the trivial $\STS_q(3)$ given by $\{V\}$, the only decided case is $\STS_2(13)$, which has been constructed in \cite{Braun-Etzion-Ostergard-Vardy-Wassermann-2013-arXiv}.

The smallest admissible case of a non-trivial $q$-Steiner triple system is $\STS_q(7)$, whose existence is open for any prime power value of $q$.
It is known as a \emph{$q$-analog of the Fano plane}, since the unique Steiner triple system $\STS_1(7)$ is the Fano plane.
It is worth noting that there are cases of Steiner systems without a $q$-analog, as the famous large Witt design with parameters $5$-$(24,8,1)$ does not have a $q$-analog for any prime power $q$ \cite{KL15}. 

\subsection{Group actions}\label{sec:groupactions}
Let $G$ be a group acting on a set $X$ via $x\mapsto x^g$.
The stabilizer of $x$ in $G$ is given by $G_x = \{g\in G\mid x^g = x\}$, and the $G$-orbit of $x$ is given by $x^G = \{x^g \mid g\in G\}$.
By the action of $G$, the set $X$ is partitoned into orbits.
For all $x\in X$, there is the correspondence $x^g \mapsto G_x g$ between the orbit $x^G$ and the set $G_x\backslash G$ of the right cosets of the stabilizer $G_x$ in $G$.
For finite orbit lengths, this implies the orbit-stabilizer theorem stating that $\#x^G = [G : G_x]$.
In particular, the orbit lengths $\#x^G$ are divisors of the group order $\# G$.

For all $g\in G$ we have
\begin{equation}
	\label{eq:stab_conj}
	G_{x^{g}} = g^{-1} G_x g\text{.}
\end{equation}
This leads to the following observations:
\begin{enumerate}
\item[(a)]\label{consequence_a}
The stabilizers of elements in the same orbit are conjugate in $G$, and any conjugate subgroup of $G_x$ is the $G$-stabilizer of some element in the $G$-orbit of $x$.
\item[(b)]\label{consequence_b} Equation~\eqref{eq:stab_conj} shows that $G_{x^g} = G_x$ for all $g\in N_G(G_x)$, where $N_G$ denotes the normalizer in $G$.
Consequentely, for any subgroup $H \leq G$ the normalizer $N_G(H)$ acts on the elements of $x\in X$ with $N_x = H$.
\end{enumerate}

The above observations greatly benefit our original problem, which is the investigation of all the subgroups $H$ of $G = \GL(7,2)$ for the existence of a binary $q$-analog $D$ of the Fano plane whose stabilizer $G_D$ equals $H$:
By observation~\ref{consequence_a}, we may restrict the search to representatives of subgroups of $G$ up to conjugacy.
Furthermore, having fixed some subgroup $H$, by observation~\ref{consequence_b} the normalizer $N = N_G(H)$ is acting on the solution space.
Consequently, we can notably speed up the search process by applying isomorph rejection with resprect to the action of $N$.

\subsection{The method of Kramer and Mesner}\label{sec:km}
The method of Kramer and Mesner \cite{KM76} is a powerful tool for the computational construction of combinatorial designs.
It has been successfully adopted and used for the construction of designs over a finite field \cite{Braun-Kerber-Laue-2005,Miyakawa-Munemasa-Yoshiara-1995}.
For example, the hitherto only known $q$-analog of a Steiner triple system in \cite{Braun-Etzion-Ostergard-Vardy-Wassermann-2013-arXiv} has been constructed by this method.
Here we give a short outline, for more details we refer the reader to \cite{Braun-Kerber-Laue-2005}. 
The \emph{Kramer-Mesner matrix} $M_{t,k}^G$ is defined to be the matrix whose rows and columns are indexed by the $G$-orbits on the set $\gauss{V}{t}{q}$ of $t$-subspaces and on the set $\gauss{V}{k}{q}$ of $k$-subspaces of $V$, respectively.
The entry of $M_{t,k}^G$ with row index $T^G$ and column index $K^G$ is defined as $\#\{K'\in K^G \mid T \leq K'\}$.
Now there exists a $G$-invariant $t$-$(v,k,\lambda)_q$ design if and only if there is a zero-one solution vector $\mathbf{x}$ of the linear equation system
\begin{equation}\label{eq:km}
M_{t,k}^{G}\mathbf{x}=\lambda \mathbf{1},
\end{equation} 
where $\textbf{1}$ denotes the all-one column vector. 
More precisely, if $\mathbf{x}$ is a zero-one solution vector of the system~\eqref{eq:km}, a $t$-$(v,k,\lambda)_q$ design is given by the union of all orbits $K^G$ where the corresponding entry in $\mathbf{x}$ equals one.
If $\mathbf{x}$ runs over all zero-one solutions, we get all $G$-invariant $t$-$(v,k,\lambda)_q$ designs in this way.

\section{Automorphisms of order $3$}
\label{sect:o3}
In this section, automorphisms of order $3$ of binary $q$-analogs of Steiner triple systems are investigated.
While the techniques are not restricted to $q=2$ or order~$3$, we decided to stay focused on our main case of interest.
In parts, we follow \cite[Section~3]{BKN16} where automorphisms of order $2$ have been analyzed.

We will assume that $V = \GF(2)^v$, allowing us to identify $\GL(V)$ with the matrix group $\GL(v,2)$.

\begin{lemma}
	\label{lem:o3type}
	In $\GL(v,2)$, there are exactly $\lfloor v/2\rfloor$ conjugacy classes of elements of order $3$.
	Representatives are given by the block-diagonal matrices $A_{v,f}$ with $f\in\{0,\ldots,v-1\}$ and $v-f$ even, consisting of $\frac{v-f}{2}$ consecutive $2\times 2$ blocks $\left(\begin{smallmatrix}0 & 1 \\ 1 & 1\end{smallmatrix}\right)$, followed by a $f\times f$ unit matrix.
\end{lemma}

\begin{proof}
Let $A\in\GL(v,2)$ and $m_A\in\GF(2)[X]$ be its minimal polynomial.
The matrix is of order $3$ if and only if $m_A$ divides $X^3 - 1 = (X+1)(X^2+X+1)$ but $m_A \neq X+1$.
Now the enumeration of the possible rational normal forms of $A$ yields the stated classification.
\end{proof}

For a matrix $A$ of order $3$, the unique conjugate $A_{v,f}$ given by Lemma~\ref{lem:o3type} will be called the \emph{type} of $A$.
The action of $\langle A_{v,f}\rangle$ partitions the point set $\gauss{\GF(2)^v}{1}{2}$ into orbits of size $1$ or $3$.
An orbit of length $3$ may either consist of three collinear points (\emph{orbit line}) or of a triangle (\emph{orbit triangle}).

\begin{lemma}
	\label{lem:o3points}
	The action of $\langle A_{v,f}\rangle$ partitions $\gauss{\GF(2)^v}{1}{2}$ into
	\begin{enumerate}
		\item[(i)] $2^f - 1$ fixed points;
		\item[(ii)] $\frac{2^{v-f} - 1}{3}$ orbit lines;
		\item[(iii)] $\frac{(2^{v-f} - 1)(2^f - 1)}{3}$ orbit triangles.
	\end{enumerate}
\end{lemma}

\begin{proof}
	Let $G = \langle A_{v,f}\rangle$.
	The eigenspace of $A_{v,f}$ corresponding to the eigenvalue $1$ is of dimension $f$ and equals $F = \langle \vek{e}_{v-f+1}, \vek{e}_{v-f+2},\ldots, \vek{e}_v\rangle$.
	The fixed points are exactly the $2^f-1$ elements of $\gauss{F}{1}{2}$.
	Furthermore, for a non-zero vector $\vek{x}\in\GF(2)^v$ the orbit $\langle\vek{x}\rangle_{\GF(2)}^G$ is an orbit line if and only if $A_{v,f}^2\vek{x} + A_{v,f}\vek{x} + \vek{x} = \vek{0}$ or equivalently,
	\[
		\vek{x} \in K := \ker(A_{v,f}^2 + A_{v,f} + I_v) = \langle\vek{e}_1,\vek{e}_2,\ldots,\vek{e}_{v-f}\rangle\text{.}
	\]
	Thus, the number of orbit lines is $\gauss{\dim(K)}{1}{2}/3 = (2^{v-f} - 1)/3$.
	The remaining $\gauss{v}{1}{2} - \gauss{f}{1}{2} - \gauss{v-f}{1}{2} = (2^{v-f}-1)(2^f-1)$ points are partitioned into orbit triangles.
\end{proof}

\begin{example}
	\label{ex:o3v3}
	We look at the classical Fano plane as the points and lines in $\PG(2,2) = \PG(\GF(2)^3)$.
	Its automorphism group is $\GL(3,2)$.
	By Lemma~\ref{lem:o3type}, there is a single conjugacy class of automorphisms of order $3$, represented by
	\[
		A_{3,1}
		= \begin{pmatrix}
			0 & 1 & 0 \\
			1 & 1 & 0 \\
			0 & 0 & 1
		\end{pmatrix}\text{.}
	\]
	By Lemma~\ref{lem:o3points}, the action of $\langle A_{3,1}\rangle$ partitions the point set $\gauss{\GF(2)^3}{1}{2}$ into the fixed point
	\[
		\langle(0,0,1)\rangle_{\GF(2)}\text{,}
	\]
	the orbit line
	\[
		\{
			\langle (1,0,0)\rangle_{\GF(2)}\text{,}\quad
			\langle (0,1,0)\rangle_{\GF(2)}\text{,}\quad
			\langle (1,1,0)\rangle_{\GF(2)}
		\}\text{,}
	\]
	and the orbit triangle
	\[
		\{
			\langle (1,0,1)\rangle_{\GF(2)}\text{,}\quad
			\langle (0,1,1)\rangle_{\GF(2)}\text{,}\quad
			\langle (1,1,1)\rangle_{\GF(2)}
		\}\text{.}
	\]
\end{example}

Now we look at planes $E$ fixed under the action of $\langle A_{v,f}\rangle$.
Here, the restriction of the automorphism $\vek{x} \mapsto A_{v,f}\vek{x}$ to $E$ yields an automorphism of $E \equiv \GF(2)^3$ whose order divides $3$.
If its order is $1$, then $E$ consists of $7$ fixed points and we call $E$ of \emph{type 7}.
Otherwise, the order is $3$.
So, by Example~\ref{ex:o3v3} it is of type $A_{3,1}$, and $E$ consists of $1$ fixed point, $1$ orbit line and $1$ orbit triangle.
Here, we call $E$ of \emph{type 1}.

\begin{lemma}
	\label{lem:o3fixedplanes}
	Under the action of $\langle A_{v,f}\rangle$,
	\begin{align*}
		\#\text{fixed planes of type }7
		& = \gauss{f}{3}{2} = \frac{(2^f-1)(2^{f-1}-1)(2^{f-2}-1)}{21}\text{;} \\
		\#\text{fixed planes of type }1
		& = \#\text{orbit triangles} = \frac{(2^f-1)(2^{v-f}-1)}{3}\text{.}
	\end{align*}
\end{lemma}

\begin{proof}
	The fixed planes of type $7$ are precisely the planes in the space of all fixed points of dimension $f$.
	Each fixed plane of type $3$ is uniquely spanned by an orbit triangle.
\end{proof}

\begin{example}
	\label{ex:o3v7}
	By Lemma~\ref{lem:o3type}, the conjugacy classes of elements of order $3$ in $\GL(7,2)$ are represented by
\begin{align*}
A_{7,1}
& = \left(\begin{smallmatrix}
    0& 1& 0& 0& 0& 0& 0 \\
    1& 1& 0& 0& 0& 0& 0 \\
    0& 0& 0& 1& 0& 0& 0 \\
    0& 0& 1& 1& 0& 0& 0 \\
    0& 0& 0& 0& 0& 1& 0 \\
    0& 0& 0& 0& 1& 1& 0 \\
    0& 0& 0& 0& 0& 0& 1 
\end{smallmatrix}\right)\text{,}
& A_{7,3}
& = \left(\begin{smallmatrix}
    0& 1& 0& 0& 0& 0& 0 \\
    1& 1& 0& 0& 0& 0& 0 \\
    0& 0& 0& 1& 0& 0& 0 \\
    0& 0& 1& 1& 0& 0& 0 \\
    0& 0& 0& 0& 1& 0& 0 \\
    0& 0& 0& 0& 0& 1& 0 \\
    0& 0& 0& 0& 0& 0& 1 
\end{smallmatrix}\right)\text{,}
& A_{7,5}
& = \left(\begin{smallmatrix}
    0& 1& 0& 0& 0& 0& 0 \\
    1& 1& 0& 0& 0& 0& 0 \\
    0& 0& 1& 0& 0& 0& 0 \\
    0& 0& 0& 1& 0& 0& 0 \\
    0& 0& 0& 0& 1& 0& 0 \\
    0& 0& 0& 0& 0& 1& 0 \\
    0& 0& 0& 0& 0& 0& 1 
\end{smallmatrix}\right)\text{.}
\end{align*}
By Lemma~\ref{lem:o3points} and Lemma~\ref{lem:o3fixedplanes}, we get the following numbers:
\[
	\begin{array}{l|rrr}
		                                & A_{7,1} & A_{7,3} & A_{7,5} \\
		\hline
		\#\text{fixed points}           &       1 &       7 &      31 \\
		\#\text{orbit lines}            &      21 &       5 &       1 \\
		\#\text{orbit triangles}        &      21 &      35 &      31 \\
		\#\text{fixed planes of type }7 &       0 &       1 &     155 \\
		\#\text{fixed planes of type }1 &      21 &      35 &      31
	\end{array}
\]
\end{example}

In the following, $D$ denotes an $\STS_2(v)$ with an automorphism $A_{v,f}$ of order $3$.
From the admissibility we get $v \equiv 1,3\mod 6$ and hence $f$ odd.
The fixed points are given by the $1$-subspaces of the eigenspace of $A_{v,f}$ corresponding to the eigenvalue $1$, which will be denoted by $F$.
The set of fixed planes in $D$ of type $7$ and $1$ will be denoted by $F_7$ and $F_1$, respectively.

\begin{lemma}
	\label{lem:fixline_block}
	Let $L\in \gauss{V}{2}{2}$ be a fixed line.
	Then the block passing through $L$ is a fixed block.
\end{lemma}

\begin{proof}
	From the design property, there is a unique block $B\in D$ passing through $L$.
	For all $A\in \langle A_{v,f}\rangle$, we have $B\cdot A\in D$ and $B\cdot A > L\cdot A = L$, so $B\cdot A = B$ by the uniqueness of $B$.
	Hence $B$ is a fixed block.
\end{proof}

\begin{lemma}
	\label{lem:o3restricteddesign}
	The blocks in $F_7$ form an $\STS_2(f)$ on $F$.
\end{lemma}

\begin{proof}
	Obviously, each fixed block of type $7$ is contained in $F$.
	Let $L \in \gauss{F}{2}{2}$.
	By Lemma~\ref{lem:fixline_block}, there is a unique fixed block $B\in D$ passing through $L$.
	Since $L$ consists of $3$ fixed points, $B$ must be of type $7$.
	Hence $B \leq F$.
\end{proof}

The admissibility of $\STS_2(f)$ yields $f\equiv 1,3\equiv 6$, so:

\begin{corollary}
	\label{cor:o3fequiv2}
	An $\STS_2(v)$ does not have an automorphism of order $3$ of type $A_{v,f}$ with $f \equiv 2\bmod 3$.
\end{corollary}

In particular, a binary $q$-analog of the Fano plane does not have an automorphism of order $3$ and type $A_{7,5}$.
This gives a theoretical confirmation of the computational result of \cite{BKN16}, where the group $\langle A_{7,5}\rangle$ 
has been excluded computationally.

\begin{lemma}
	\label{lem:o3fixedblocks}
	\begin{align}
		\#F_7 & = \frac{(2^f - 1)(2^{f-1}-1)}{21}\text{;} \label{eq:f7} \\
		\#F_1 & = \#\text{orbit lines} = \frac{2^{v-f}-1}{3} \label{eq:f1} \text{.}
	\end{align}
\end{lemma}

\begin{proof}
	By Lemma~\ref{lem:o3restricteddesign}, the number $\#F_7$ equals the $\lambda_0$-value of an $\STS_2(f)$.

	For $\#F_1$, we double count the set $X$ of all pairs $(L,B)$ where $L$ is an orbit line, $B\in F_1$ and $L < B$.
	By Lemma~\ref{lem:o3points}, the number of choices for $L$ is $\frac{2^{v-f} - 1}{3}$.
	Lemma~\ref{lem:fixline_block} yields a unique fixed block $B$ passing through $L$.
	Since $B$ contains the orbit line $L$, $B$ has to be of type $1$.
	So $\#X = \frac{2^{v-f} - 1}{3}$.
	On the other hand, there are $\#F_1$ possibilities for $B$ and each such $B$ contains a single orbit line.
	So $\#X = \#F_1$, verifying Equation~\eqref{eq:f1}.
\end{proof}

\begin{lemma}
	\label{lem:o3feqivv}
	An $\STS_2(v)$ with $v \geq 7$ does not have an automorphism of order $3$ of type $A_{v,f}$ with $f > (v-3)/2$ and $f \nequiv v \mod 3$.
\end{lemma}

\begin{proof}
	Assume that $v\geq 7$ and $f \nequiv v \mod 3$.
	Let $P\in \gauss{F}{1}{2}$ and $X$ be the set of all blocks passing through $P$ which are not of type $7$.
	The number of blocks passing through $P$ is $\lambda_1 = \frac{2^{v-1} - 1}{3}$.
	By Lemma~\ref{lem:o3restricteddesign}, $F_7$ is an $\STS_2(f)$ on $F$.
	So the number of blocks of type $7$ passing through $P$ is given by the $\lambda_1$-value of an $\STS_2(f)$, which equals $\frac{2^{f-1} - 1}{3}$.
	Hence $\#X = \frac{2^{v-1} - 2^{f-1}}{3}$.
	Since $P$ is a fixed point, the action of $\langle A_{v,f}\rangle$ partitions $X$ into orbits of size $1$ and $3$.
	Depending on $v$ and $f$, the remainder of $\#X$ modulo $3$ is shown below:
	\[
		\begin{array}{c|ccc}
			                 & f\equiv 1\bmod 6 & f\equiv 3\bmod 6 & f\equiv 5\bmod 6 \\
					 \hline
			v\equiv 1\bmod 6 &                0 &                1 &                2 \\
			v\equiv 3\bmod 6 &                2 &                0 &                1
		\end{array}
	\]
	In our case $f\nequiv v\bmod 3$, we see that $\#X$ is not a multiple of $3$, implying the existence of at least one fixed block in $X$, which must be of type $1$.
	Thus, it contains only $1$ fixed point, showing that the type $1$ blocks coming from different points $P\in \gauss{F}{1}{2}$ are pairwise distinct.
	In this way, we see that
	\[
	2^f - 1
	= \#\text{fixed points}
	\leq \#F_1
	= \frac{2^{v-f}-1}{3}\text{,}
	\]
	where the last equality comes from Lemma~\ref{lem:o3fixedblocks}.
	Using the preconditions $v \geq 7$ and $v,f$ odd, we get that this inequality is violated for all $f > (v-3)/2$.
\end{proof}

\begin{remark}
	\begin{enumerate}[(a)]
		\item The condition $v \geq 7$ cannot be dropped since the automorphism group of the trivial $\STS_2(3)$ is the full linear group $\GL(3,2)$ containing an automorphism of type $A_{3,1}$.
		\item In the case that the remainder of $\#X$ modulo $3$ equals $2$, we could use the stronger inequality $2(2^f - 1) \leq \#F_1$.
		However, the final condition on $f$ is the same.
	\end{enumerate}
\end{remark}

Lemma~\ref{lem:o3feqivv} allows us to exclude one of the groups left open in~\cite[Theorem~1]{BKN16}:

\begin{corollary}
	There is no binary $q$-analog of the Fano plane invariant under $G_{3,2} := \langle A_{7,3}\rangle$.
\end{corollary}

As a combination of Lemma~\ref{lem:o3type}, Corollary~\ref{cor:o3fequiv2} and Lemma~\ref{lem:o3feqivv}, we get:

\begin{theorem}
	\label{thm:o3}
	Let $D$ be an $\STS_2(v)$ with an automorphism $A$ of order $3$.
	Then $A$ has the type $A_{v,f}$ with $f\nequiv 2\mod 3$.
	If $f \equiv v\mod 3$, then either $v=3$ or $f \leq (v-3)/2$.
\end{theorem}

\begin{example}
	Theorem~\ref{thm:o3} excludes about half of the conjugacy types of elements of order $3$.
	Below, we list the remaining ones for small admissible values of $v$:
\[
	\begin{array}{l|rrrrrr}
		                                & A_{7,1} & A_{9,1} & A_{9,3} & A_{13,1} & A_{13,3} & A_{13,7} \\
		\hline
		\#\text{fixed points}           &       1 &       1 &       7 &        1 &        7 &      127 \\
		\#\text{orbit lines}            &      21 &      85 &      21 &     1365 &      341 &       21 \\
		\#\text{orbit triangles}        &      21 &      85 &     147 &     1365 &     2387 &     2667 \\
		\#\text{fixed planes of type }7 &       0 &       0 &       1 &        0 &        1 &    11811 \\
		\#\text{fixed planes of type }1 &      21 &      85 &     147 &     1365 &     2387 &     2709 \\
		\#F_7                           &       0 &       0 &       1 &        0 &        1 &      381 \\
		\#F_1                           &      21 &      85 &      21 &     1365 &      341 &       21 
	\end{array}
\]
\end{example}

We conclude this section with an investigation of the case $A_{v,1}$, which has not been excluded for any value of $v$.
The computational treatment of the open case $A_{7,1}$ in Section~\ref{sect:comp} will make use of the structure result of the following lemma.

\begin{lemma}
	Let $D$ be a $\STS_2(v)$ with an automorphism of type $A_{v,1}$.
	Then $D$ contains $\frac{2^{v-1} - 1}{3}$ fixed blocks of type $1$.
	The remaining blocks of $D$ are partitioned into orbits of length $3$.
	Furthermore, $V$ can be represented as $V = W + X$ with $\GF(2)$ vector spaces $W$ and $X$ of dimension $v-1$ and $1$, respectively, such that the fixed blocks of type $1$ are given by the set $\{L + X : L \in \mathcal{L}\}$, where $\mathcal{L}$ is a Desarguesian line spread of $\PG(W)$.
\end{lemma}

\begin{proof}
	Let $W = \GF(2^{v-1})$, which will be considered as a $\GF(2)$ vector space if not stated otherwise.
	Let $\zeta\in W$ be a primitive third root of unity.
	We consider the automorphism $\varphi : \vek{x} \mapsto \zeta\vek{x}$ of $W$ of order $3$.
	Since $\varphi$ does not have fixed points in $\gauss{W}{1}{2}$, $\varphi$ is of type $A_{v-1,0}$.
	The set $\mathcal{L} = \gauss{W}{1}{4}$ is a Desarguesian line spread of $\PG(W)$.
	It consists of all lines of $\PG(W)$ with $\varphi(L) = L$.
	Since $\PG(W)$ does not contain any fixed points under the action of $\varphi$, $\mathcal{L}$ is the set of the $(2^{f-1} - 1)/3$ orbit lines.

	Now let $X$ be a $\GF(2)$ vector space of dimension $1$.
	The map $\hat{\varphi} = \varphi \times\id_X$ is an automorphism of $V = W\times X$ of order $3$ and type $A_{v,1}$.
	Let $\hat{\mathcal{L}} = \{L + X \mid L\in\mathcal{L}\}$.
	Under the action of $\hat{\varphi}$, the elements of $\hat{\mathcal{L}}$ are fixed planes of type 1.
	By Lemma~\ref{lem:o3fixedplanes}, the total number of fixed planes of type $1$ equals $\#\hat{\mathcal{L}} = \#\mathcal{L}$, so $\hat{\mathcal{L}}$ is the full set of fixed planes of type $1$.
	Moreover, Lemma~\ref{lem:o3fixedblocks} gives $\#F_1 = (2^{f-1} - 1)/3 = \#\hat{\mathcal{L}}$, on the one hand, so all these planes have to be blocks of $D$, and $\#F_7 = 0$ on the other hand, so the remaining blocks are partitioned into orbits of length $3$.
\end{proof}

\section{Computational results}
\label{sect:comp}
The automorphism groups $G_{3,1}$ and $G_{4}$ of a tentative $\STS_2(7)$ are excluded computationally by the method 
of Kramer and Mesner from Section~\ref{sec:km}. The matrix $M_{t,k}^{G_4}$ consists of $693$ rows and $2439$ columns, the 
matrix $M_{t,k}^{G_{3,1}}$ has $903$ rows and $3741$ columns. In both cases, columns containing entries larger than $1$ had 
been ignored since from equation~(\ref{eq:km}) it is immediate that the corresponding $3$-orbits cannot be part of a Steiner system. 

One of the fastest method for exhaustively searching all $0/1$ solutions of such a system of linear equations where all 
coefficients are in $\{0, 1\}$ is the backtrack algorithm \emph{dancing links}~\cite{Knuth:01}. 
We implemented a parallel version of the algorithm which is well suited 
to the job scheduling system \emph{Torque} of the Linux cluster of the University of Bayreuth. 
The parallelization approach is straightforward:
In a first step all paths of the dancing links algorithm down to a certain level are stored. In the second step every such path is started as a 
separate job on the computer cluster, where initially the algorithm is forced to start with the given path.

For the group $G_{4}$ the search was divided into $192$ jobs. 
All of these determined that there is no $\STS_2(7)$ with automorphism group $G_4$. 
Together, the exhaustive search of all these $192$ sub-problems took approximately $5500$ CPU-days. 

The group $G_{3,1}$ was even harder to tackle. The estimated run time (see \cite{Knuth:01}) for 
this problem is $27\,600\,000$ CPU-days.

In order to break the symmetry of this search problem and avoid unnecessary computations, 
the normalizer $N(G_{3,1})$ of $G_{3,1}$ in $\GL(7,2)$ proved to be useful. 
According to GAP \cite{GAP4}, the normalizer is generated by
$$N(G_{3,1}) = 
\left\langle
	\left(\begin{smallmatrix}
	 0& 1& 0& 0& 0& 0& 0\\
     1& 1& 0& 0& 0& 0& 0\\
     0& 0& 0& 1& 0& 0& 0\\
     0& 0& 1& 1& 0& 0& 0\\
     0& 0& 0& 0& 0& 1& 0\\
     0& 0& 0& 0& 1& 1& 0\\
     0& 0& 0& 0& 0& 0& 1\\
    \end{smallmatrix}\right),
	\left(\begin{smallmatrix}
	 0& 1& 0& 1& 0& 1& 0\\
     1& 0& 1& 0& 1& 0& 0\\
     1& 1& 0& 0& 0& 0& 0\\
     0& 1& 0& 0& 0& 0& 0\\
     0& 0& 0& 0& 1& 1& 0\\
     0& 0& 0& 0& 0& 1& 0\\
     0& 0& 0& 0& 0& 0& 1\\
    \end{smallmatrix}\right),
	\left(\begin{smallmatrix}
     0& 1& 0& 1& 0& 1& 0\\
     1& 1& 1& 1& 1& 1& 0\\
     1& 0& 0& 0& 0& 0& 0\\
     0& 1& 0& 0& 0& 0& 0\\
     0& 0& 1& 0& 0& 0& 0\\
     0& 0& 0& 1& 0& 0& 0\\
     0& 0& 0& 0& 0& 0& 1\\
    \end{smallmatrix}\right)
\right\rangle
$$
and has order $362\,880$.

As discussed in Section~\ref{sec:groupactions}, if for a prescribed group $G$, $s_1, s_2$ are two solutions of the Kramer-Mesner equations (\ref{eq:km}), then $s_1$ and 
$s_2$ correspond to two designs $D_1$ and $D_2$ both having $G$ as full automorphism group.
A permutation $\sigma_n$ which maps the $1$-entries of $s_1$ to the $1$-entries of $s_2$ can be represented by
an element $n\in \GL(7,2)$. In other words, $D_1^n = D_2$. 
Since $G$ is the full automorphism group of $D_1$ and $D_2$ it follows for all $g\in G$:
$$
	D_1^{ng} = D_2^{g} = D_2 = D_1^{n}.
$$ 
This shows that $n\in N(G)$.

This can be used as follows in the search algorithm.
We force one orbit $K_i^G$ to be in the design. If dancing links shows that there is no solution which contains this
orbit, all $k$-orbits in $(K_1^G)^N$ can be excluded from being part of a solution, i.e. the corresponding columns
of $M_{t,k}^G$ can be removed.

In the case $G_{3,1}$, the set of $k$-orbits is partitioned into four orbits under the normalizer $N(G_{3,1})$.
Two of this four orbits, let's call them $K_1^G$ and $K_2^G$, can be excluded with dancing links in a few seconds. 
The third orbit $K_3^G$ needs more work, see below. 
After excluding the third orbit, also the fourth orbit is excluded in a few seconds.

For the third orbit $K_3^G$ we iterate this approach and fix two $k$-orbits simultaneously, one of them being $K_3^G$.  
That is, we consider all cases of fixed pairs $(K_3^G, K_i^G)$, where $K_i^G \notin (K_1^G)^N\cup(K_2^G)^N$.
If there is no design which contains this pair of $k$-orbits, all $k$-orbits of the orbit $(K_i^G)^S$ can be excluded too,
where $S=G_{K_3^G}$ is the stabilizer of the orbit $K_3^G$ under the action of $N(G)$.

This process could be repeated for triples, but run time estimates show that fixing pairs of $k$-orbits minimizes the computing 
time.\footnote{If iterated till the end, this type of search algorithm is known as \emph{orderly generation}, see 
e.g.~\cite{winner,Royle1998105}.}
Under the stabilizer of $K_3^G$, the set of pairs $(K_3^G, K_i^G)$ of $k$-orbits is partitioned into $14$ orbits.
Seven of these $14$ pairs representing the orbits lead to problems which could be solved in a few seconds.
The remaining seven sub-problems were split into $49\,050$ separate jobs with the above approach for parallelization. 
These jobs could be completed by dancing links 
in approximately $23\,600$ CPU-days on the computer cluster, 
determining that there is no $\STS_2(7)$ with automorphism group $G_{3,1}$. 

For the group $G_2$ the estimated run time is $3\,020\,000\,000\,000\,000$
CPU-days which seems out of reach with the methods of this paper.

%

\section*{Acknowledgements}
The authors would like to acknowledge the financial support provided by COST -- \textit{European  Cooperation  in  Science  and  
Technology} -- within the Action IC1104 \textit{Random Network Coding and Designs over $GF(q)$}. 
The authors also want to thank the \emph{IT service center} of the 
University Bayreuth for providing the excellent computing cluster, and especially Dr. Bernhard Winkler for his support.

\end{document}